\documentclass{amsart}

\usepackage{amssymb}
\usepackage{enumerate}
\usepackage{url}

\newtheorem{theorem}{Theorem}[section]

\newtheorem{lemma}[theorem]{Lemma}

\newtheorem{corollary}[theorem]{Corollary}

\newtheorem{definition}[theorem]{Definition} 

\newtheorem{notation}[theorem]{Notation}

\begin{document}

\title[Gibbs measure at high temperature]{Alternative proof of existence of Gibbs measure at high temperature}

\author[F. Kachapova]{Farida Kachapova}
\address{Department of Mathematical Sciences
\\Auckland University of Technology
\\New Zealand}

\email{farida.kachapova@aut.ac.nz}

\author[I. Kachapov]{Ilias Kachapov}
\address{Examination Academic Services\\ University of Auckland\\ New Zealand}
\email{kachapov@xtra.co.nz}

\date{\today}

\begin{abstract}
Mathematical models in equilibrium statistical mechanics describe physical systems with many particles interacting with an external force and with one another. Gibbs measure is a fundamental concept in this theory. In existing literature infinite-volume models are constructed as limits of finite models and existence of Gibbs measure for them is proven through DLR formalism. The general existence proofs are quite complicated and involve topology and cluster expansion.

In this paper we develop a more transparent and more constructive proof of existence of infinite Gibbs measure for a particular case of  interaction model at high temperature. The proof is based on a limiting procedure and involves estimates of series of semi-invariants and graph-related estimates. 
\end{abstract}



\maketitle

\section{Introduction}
Mathematical models of statistical mechanics are described in many publications; see, for example, \cite{D68},  \cite{K77}, \cite{MM91}, \cite{Ma80}, \cite{B08} and \cite{Yan}. A generalised model of a physical system with many particles is based on a probability space. When an infinite model is constructed as a limit of finite models, there occurs a problem of existence of its probability measure (Gibbs measure). The general mathematical proof of existence of Gibbs measure was constructed using DLR approach. It was described in several books: \cite{Rue}, \cite{Pres}, \cite{MM91}, \cite{FV17}, and in most general form in \cite{Geor}. The general existence proof is quite complicated and involves topology, cluster expansion, and conditional probabilities. 

In this paper we provide a straightforward existence proof for a less general model. 
In \cite{K15} and \cite{K17} we studied a concept of an interaction model. The interaction model describes an equilibrium physical system at high temperature with many particles interacting with an external force and with one another. In this paper  we improve the concept of interaction model by giving a more general and detailed definition. For the infinite case of this model we prove existence of Gibbs measure as a limiting case of probability measures on the finite models. Our proof involves only basic probability, calculus and combinatorics, and is more transparent and constructive than the general proof. We believe that the constructions in this paper are of interest on their own even though the existence of infinite Gibbs measure is a well-known fact.

In Section \ref{section:finite} we 
introduce a probability space and characteristics for a finite interaction model, and give a definition of this model. In Section \ref{section:infinite} we use some estimations in graphs to describe one of interaction characteristics in a physical system (Lemma \ref{L exists}). Next in this section we state Convergence Theorem and use it to construct an infinite interaction model; then we briefly justify that the Ising and Potts models are particular cases of the interaction model, and its generalization includes the $n$-model. 

The remaining Sections \ref{section:estimations} and \ref{section:main_proof} give a detailed proof of the Convergence Theorem. In particular, in Section \ref{section:estimations} we produce more graph-related estimations and apply them to semi-invariants and their series. In Section \ref{section:main_proof} we show that the probability $P_N(A)$ of an event $A$ in a finite model can be written as a series of semi-invariants. We finalise the proof of 
the Convergence Theorem in Section \ref{section:main_proof} by proving absolute and uniform convergence of the aforementioned series for 
$P_N(A)$ and taking a limit as $N\rightarrow\infty$. 

\section{Finite Interaction Model}
\label{section:finite}

\subsection{Gibbs Modification}
Gibbs modification is used 
to modify a given probability in order to describe interaction between particles.

For any probability measure $P$ on $(\Omega,\Sigma)$ denote $\langle\cdot\rangle_P$ the expectation with respect to $P$. 
\begin{definition}
Suppose $P$ is a probability measure on $(\Omega,\Sigma)$ and $U$ is a bounded random variable on $(\Omega,\Sigma)$.

\textbf{Gibbs modification} of the probability measure $P$ by the random variable $U$ is denoted $P_U$ and is defined as follows.
For any event $A\in\Sigma$:
\[P_{U}(A)= 
\frac{\langle I_A \,e^{U}\rangle_P}{\langle e^{U}\rangle_P},\]
where $I_A$ denotes the indicator of the event $A$.
\end{definition}

\begin{lemma}
In conditions of the previous definition:
\begin{enumerate}[1)]
\item $P_U$ is a probability measure on $(\Omega,\Sigma)$; 

\item for any random variable $Y$ on $(\Omega,\Sigma)$, $\langle Y\rangle_{P_U}=\dfrac{\langle Y \,e^{U}\rangle_P}{\langle e^{U}\rangle_P}$.
\end{enumerate}
\label{lemma:modification}
\end{lemma}

Proof is well-known and can be found, for example, in \cite{K17}.

\subsection{Construction of Probability Space}

\begin{definition}
We introduce some objects that will be used to construct an interaction model.
\begin{enumerate}[1)]
\item Fix a natural number $\nu\geqslant1$ and consider a $\nu$-dimensional integer lattice:
\[\mathbb{Z}^\nu= \{\boldsymbol{t}=(t_1,\ldots,t_\nu)\mid t_i\in \mathbb{Z}, i=1,\ldots,\nu\}\]
with the distance between any two points $\boldsymbol{s},\boldsymbol{t}\in\mathbb{Z}^\nu$ defined by:
\[\|\boldsymbol{s}-\boldsymbol{t}\|=\sum_{i=1}^{\nu} |s_i-t_i|.\]

\item For any $\boldsymbol{t}\in\mathbb{Z}^\nu$ we denote $\Omega_{\boldsymbol{t}}=\mathbb{R}$ and $\Sigma_{\boldsymbol{t}}$ the Borel $\sigma$-algebra on $\mathbb{R}$; we fix a probability measure $P_{\boldsymbol{t}}$ on $(\Omega_{\boldsymbol{t}},\Sigma_{\boldsymbol{t}})$.

We call $\{P_{\boldsymbol{t}},\boldsymbol{t}\in\mathbb{Z}^\nu\}$ the \textbf{initial probabilities}.
\end{enumerate}
\label{def:P_t}
\end{definition}

\begin{definition}
\begin{enumerate}[1)]
\item By the Hahn-Kolmogorov theorem, for any integer \\$k>1$ and distinct $\boldsymbol{t}_1,\ldots,\boldsymbol{t}_k\in\mathbb{Z}^\nu$ there exists a unique probability space 
\\$(\Omega_{\boldsymbol{t}_1}\times\ldots\times\Omega_{\boldsymbol{t}_k},\Sigma_{\boldsymbol{t}_1\ldots\boldsymbol{t}_k} ,P_{\boldsymbol{t}_1\ldots\boldsymbol{t}_k})$, where $\Sigma_{\boldsymbol{t}_1\ldots\boldsymbol{t}_k}$ is the $\sigma$-algebra generated by elements of the Cartesian product $\Sigma_{\boldsymbol{t}_1}\times\ldots\times\Sigma_{\boldsymbol{t}_k}$ and the probability measure $P_{\boldsymbol{t}_1\ldots,\boldsymbol{t}_k}$ has the property:
\[P_{\boldsymbol{t}_1\ldots\boldsymbol{t}_k}(F_1\times\ldots\times F_k)=P_{\boldsymbol{t}_1}(F_1)\times\ldots\times P_{\boldsymbol{t}_k}(F_k)\]
for any $F_i\in\Sigma_{\boldsymbol{t}_i}\;(i=1,\ldots,k)$.
 
\item By the Kolmogorov extension theorem, there exists a probability space $(\Omega,\Sigma, P_0)$, where the elements are described as follows. 
\begin{enumerate}[(a)]
\item $\Omega=\mathbb{R}^{\mathbb{Z}^\nu}$, that is $\Omega
=\left\lbrace 
\omega\mid\omega:\mathbb{Z}^\nu\rightarrow \mathbb{R}\right\rbrace$. 

An element $\omega$ of $\Omega$ is called a  \textbf{configuration} and is interpreted as a state of a physical system.

\item $\Sigma$ is the $\sigma$-algebra generated by sets of the form 
$\{\omega\in\Omega\mid\omega(\boldsymbol{t})\in F\}$ for all 
$\boldsymbol{t}\in\mathbb{Z}^\nu$ and $F\in\Sigma_{\boldsymbol{t}}$.
\item $P_0$ is a unique probability measure on $(\Omega,\Sigma)$ such that for any integer $k>1$, distinct $\boldsymbol{t}_1,\ldots,\boldsymbol{t}_k\in\mathbb{Z}^\nu$ and  $F_1\in\Sigma_{\boldsymbol{t}_1},\ldots,F_k\in\Sigma_{\boldsymbol{t}_k}$:
\[P_0(\omega(\boldsymbol{t}_1)\in F_1,\ldots,\omega(\boldsymbol{t}_k) \in F_k)
=P_{\boldsymbol{t}_1\ldots\boldsymbol{t}_k}(F_1\times\ldots\times F_k).\]

Thus, $P_0$ is uniquely defined by the initial probabilities $\{P_{\boldsymbol{t}},\boldsymbol{t}\in\mathbb{Z}^\nu\}$.
\end{enumerate}

\item For any $\boldsymbol{t}\in\mathbb{Z}^\nu$ a function $X_{\boldsymbol{t}}:\Omega\rightarrow \mathbb{R}$ is defined by the following:
\[X_{\boldsymbol{t}}(\omega)=\omega(\boldsymbol{t}).\]
$X$ denotes $\{X_{\boldsymbol{t}}\mid \boldsymbol{t}\in\omega\}$.
\end{enumerate}
\label{def:initial}
\end{definition}

For the rest of the paper we fix the objects $\mathbb{Z}^\nu$, $(\Omega,\Sigma, P_0)$, and $X$ from these definitions. For brevity we denote $\langle\cdot\rangle_{P_0}$ as $\langle\cdot\rangle_0$.

\begin{lemma}
$\{X_{\boldsymbol{t}}\mid \boldsymbol{t}\in \mathbb{Z}^\nu\}$ is an independent random field on the probability space $(\Omega,\Sigma, P_0)$.
\end{lemma}
\begin{proof}
Proof follows from the definitions.
\end{proof}
\begin{definition}
\begin{enumerate}[1)]
\item Consider a \textbf{graph} $(V,E)$, where the set of vertices $V$ is a finite subset of $\mathbb{Z}^\nu$ and $E$ is the set of edges; each edge can be regarded as a pair of distinct vertices (there are no loops). The \textit{length} of each edge is the distance between its end vertices.

The graph $(V,E)$ is called \textbf{1-connected} if it is connected and the length of any of its edges equals 1. 
\item For a finite set $B\subset \mathbb{Z}^\nu$ with at least two elements define its \textbf{size} $S(B)$ as the minimum number of edges of 1-connected graphs $(V,E)$ such that  $B\subseteq V$.
\end{enumerate}
\label{def:size}
\end{definition}

\textit{Remark.} For any finite set $B$ there is a cube in $\mathbb{Z}^\nu$ containing $B$; this cube with all its edges of length 1 is a 1-connected graph. So $S(B)$ is always defined.

\begin{notation}
For any finite subset $B$ of $\mathbb{Z}^\nu$ denote $\Sigma_B$ the $\sigma$-algebra generated by sets of the form 
$\{\omega\in\Omega\mid\omega(\boldsymbol{t})\in F\}$ for all 
$\boldsymbol{t}\in B$ and  $F\in\Sigma_{\boldsymbol{t}}$.
\end{notation}

\begin{definition}
Here we introduce three 
characteristics $\lambda,r,\Phi$ of an interaction model and a set $\mathfrak{B}$.
\begin{enumerate}[1)]
\item $\lambda\in\mathbb{R}$, $\lambda\geqslant0$. $\lambda$ is inversely proportional to the temperature of the physical system.
\item $r\in\mathbb{R}$, $r\geqslant1$. $r$ is called the \textbf{radius of interaction}.

\item Denote $\mathfrak{B}=\{B\subset\mathbb{Z}^\nu\,\,|\,\, 1\leqslant S(B)\leqslant r\}.$ 

\item For each $B\in\mathfrak{B}$, $\Phi_B$ is a random variable on $(\Omega,\Sigma_B)$ that satisfies the condition: 
\[|\Phi_B|\leqslant\lambda.\]
$\Phi$ denotes $\{\Phi_B\mid B\in \mathfrak{B}\}$. $\Phi$ 
is called the \textbf{potential} of the system. 
\end{enumerate}
\label{def:parameters}
\end{definition}

\textit{Remarks}. 1) Clearly, if $B\in\mathfrak{B}$, then $B$ is finite.

2) Two particles (represented by points of $\mathbb{Z}^\nu$) interact only if they both belong to some set $B$ ($B\in\mathfrak{B}$), and $\Phi_B$ is the interaction energy of all elements of $B$.

3) In our definition of $\mathfrak{B}$ we have only sets with $S(B)\geqslant 1$, i.e. sets with at least two elements. The general case (that includes one-element sets $B$) can be reduced to this case by a single Gibbs modification as shown in \cite{K17}, pg.343-344.

\subsection{Finite Interaction Model}
\label{section:finite_model}

\begin{definition}
A
\textbf{finite interaction model} with characteristics $\lambda,r,\Phi$, $\{P_{\boldsymbol{t}}\mid \boldsymbol{t}\in\mathbb{Z}^\nu\}$ is a sequence $(\Lambda,U_{\Lambda},P_{\Lambda})$ of three objects defined as follows.
\begin{enumerate}[1)]
\item
$\Lambda$ is an arbitrary finite subset of the lattice $\mathbb{Z}^\nu$. 

\item A function $U_{\Lambda}:\Omega\rightarrow\mathbb{R}$ is called the \textbf{interaction energy} of $\Lambda$ and is defined by the following:
\[\text{ for any  }\omega\in\Omega, \;U_{\Lambda} (\omega)=\sum_{B\in\mathfrak{B},B\subseteq\Lambda}\Phi_B(\omega), 
\]
where the set $\mathfrak{B}$ is defined in Definition \ref{def:parameters}.

$U_{\Lambda}(\omega)$ characterizes the energy of configuration $\omega$ in  $\Lambda$.

\item Denote $P_{\Lambda}=(P_0)_{U_{\Lambda}}$, the Gibbs modification of the probability $P_0$ by the interaction energy $U_{\Lambda}$.
\end{enumerate}  
\label{def:finite_model}
\end{definition}

In \cite{K17} it was shown that $(\Omega,\Sigma, P_{\Lambda})$ is a probability space.

\section{Infinite Interaction Model}
\label{section:infinite}

\subsection{Tracks in Graphs}

\begin{definition}
A \textbf{track} is a sequence $tr=(\boldsymbol{t}_0,\boldsymbol{t}_1,\ldots,\boldsymbol{t}_n)$ of points in $\mathbb{Z}^\nu$ such that $\boldsymbol{t}_0=\boldsymbol{t}_n$ and $\|\boldsymbol{t}_i-\boldsymbol{t}_{i-1}\|=1$, $i=1,\ldots,n$.
\end{definition}

\begin{lemma}
Fix $\boldsymbol{t}_0\in\mathbb{Z}^\nu$ and $n>1$. The number of tracks of the form $(\boldsymbol{t}_0,\boldsymbol{t}_1,\ldots,\boldsymbol{t}_n)$ is not greater than $(2\nu)^{n-1}$.
\label{lemma:track_1}
\end{lemma}

\begin{proof}
The first element $\boldsymbol{t}_0$ is fixed. There are $2\nu$ choices for element $\boldsymbol{t}_1$, at most $2\nu$ choices for element $\boldsymbol{t}_2$, at most $2\nu$ choices for element $\boldsymbol{t}_3$, $\ldots$, and only one choice for element $\boldsymbol{t}_n$, since $\boldsymbol{t}_n=\boldsymbol{t}_0$. So the total number of tracks is at most $(2\nu)^{n-1}$.
\end{proof}

\begin{definition}
Suppose $B$ is a finite subset of $\mathbb{Z}^\nu$ and $m=S(B)$. There is a 1-connected graph $(V,E)$ such that $B\subseteq V$, and the number of edges in  this graph is minimal and equals $m$. We denote $G_B$ the first (in lexicographic order) of such graphs and call it the \textbf{associated graph} of $B$.
\end{definition}

\begin{lemma}
Suppose $B$ is a finite subset of $\mathbb{Z}^\nu$ and $m=S(B)$. Then the graph $G_B$ is a tree with $m$ edges and $m+1$ vertices.
\label{lemma:assoc_graph}
\end{lemma}

\begin{proof}
A tree is a connected graph with no cycles. The definition of size implies that $G_B$ is connected. 

If $G_B$ has a cycle, then we can remove one edge of the cycle thus reducing the number of edges while the graph remains connected with the same set of vertices. That contradicts the choice of $G_B$ as the graph with minimal number of edges. 

By the definition $G_B$ has $m$ edges. Since it is a tree, the number of its vertices is $m+1$.
\end{proof}

\begin{definition}
Suppose $B$ is a finite subset of $\mathbb{Z}^\nu$ and $m=S(B)$. We create the \textbf{associated track} $tr_B$ of $B$ in the following two steps.

1. We obtain a graph $G'_B$ from the associated graph $G_B$ by adding for every its edge another edge with the same ends.

2. The new graph $G'_B$ has $2m$ edges and each of its vertices has an even degree. Therefore there is an Eulerian path in $G'_B$, i.e. a closed path that includes every edge of the graph exactly once. We denote $(e_1,e_2,\ldots,e_{2m})$ the first (in lexicographic order) of such paths. 

We define $tr_B$ as the corresponding sequence $(\boldsymbol{t}_0,\boldsymbol{t}_1,\boldsymbol{t}_2,\ldots,\boldsymbol{t}_{2m})$ of vertices, that is $\{\boldsymbol{t}_{i-1},\boldsymbol{t}_i\}$ are the ends of $e_i$ $(i=1,2,\ldots,2m)$.
\end{definition}

\begin{definition}
Suppose $B\in\mathfrak{B}$. Its associated track has the form  \\$tr_B=(\boldsymbol{t}_0,\boldsymbol{t}_1,\boldsymbol{t}_2,\ldots,\boldsymbol{t}_{2m})$, where $m=S(B)$ and $m\leqslant r$. 

We define the \textbf{extended track}
$tr'_B=(\boldsymbol{t}_0,\boldsymbol{t}_1,\boldsymbol{t}_2,\ldots,\boldsymbol{t}_{2m},\boldsymbol{t}_{2m+1},\ldots,\boldsymbol{t}_{2r})$ of $B$ by the following:
\begin{itemize}
\item for $i=1,\ldots,r-m$ each point $\boldsymbol{t}_{2m+i}$ is obtained from $\boldsymbol{t}_{2m+i-1}$
by increasing its first coordinate by 1; 
\item for $i=r-m+1,\ldots,2r-2m$ each point $\boldsymbol{t}_{2m+i}$ is obtained from $\boldsymbol{t}_{2m+i-1}$
by decreasing its first coordinate by 1.
\end{itemize}
\end{definition}

\begin{lemma}
\begin{enumerate}[1)]
\item For any finite subset $B$ of $\mathbb{Z}^\nu$, $tr_B$ is a track.
\item For any $B\in\mathfrak{B}$, $tr'_B$ is a track.
\end{enumerate}
\label{lemma:track_2}
\end{lemma}

\begin{proof}
Follows from the definitions.
\end{proof}

\subsection{Estimations in Graphs}

\begin{lemma}
Suppose $tr=(\boldsymbol{t}_0,\boldsymbol{t}_1,\boldsymbol{t}_2,\ldots,\boldsymbol{t}_{2r})$ is a track. Then there are at most $2^{2r-1}$ sets $C\in \mathfrak{B}$ such that $\boldsymbol{t}_0\in C$ and $tr'_C=tr$.
\label{lemma:estim_1}
\end{lemma}

\begin{proof}
Since the track $tr$ is a closed path, we have $\boldsymbol{t}_0=\boldsymbol{t}_{2r}$. Denote $V_1=\{\boldsymbol{t}_1,\boldsymbol{t}_2,\ldots,\boldsymbol{t}_{2r-1}\}$. If a set  $C\in \mathfrak{B}$ has $tr'_C=tr$, then
\begin{equation}
C\setminus \{\boldsymbol{t}_0\}\subseteq V_1.
\label{eq:track_1}
\end{equation}

Since $V_1$ has at most $2^{2r-1}$  different subsets, so there are at most $2^{2r-1}$ sets $C$ with property \eqref{eq:track_1}.
\end{proof}

\begin{lemma}
Suppose $\boldsymbol{t}_0\in\mathbb{Z}^\nu$. Then there are at most $(4\nu)^{2r-1}$ sets $C\in \mathfrak{B}$ with $\boldsymbol{t}_0\in C$.
\label{lemma:estim_2}
\end{lemma}
\begin{proof}
For each set $C\in \mathfrak{B}$ with $\boldsymbol{t}_0\in C$ we have an extended track $tr'_C$; since a track is a closed path, we can assume $tr'_C$ starts at $\boldsymbol{t}_0$. So there is a correspondence between sets $C\in \mathfrak{B}$ with $\boldsymbol{t}_0\in C$ and tracks of the form $(\boldsymbol{t}_0,\boldsymbol{t}_1,\ldots,\boldsymbol{t}_{2r})$. By Lemma  \ref{lemma:estim_1} at most $2^{2r-1}$ sets correspond to each track and by Lemma \ref{lemma:track_1} the total number of such tracks starting at $\boldsymbol{t}_0$ is not greater than $(2\nu)^{2r-1}$. Therefore the number of sets $C\in \mathfrak{B}$ with $\boldsymbol{t}_0\in C$ is not greater than $2^{2r-1}(2\nu)^{2r-1}=(4\nu)^{2r-1}$.
\end{proof}

\begin{definition}
\label{L}
1. For any $B\in \mathfrak{B}$ denote $l(B)$ the number of sets $C\in \mathfrak{B}$ that intersect with $B$.

2. Denote $L=max\{l(B)\mid B\in\mathfrak{B}\}$.
\end{definition}

Clearly, $L$ depends only on $\nu$ and $r$.

\begin{lemma}
In the interaction model the number $L$ is defined and 
\[L\leqslant(4\nu)^{2r-1}(r+1).\]
\label{L exists}
\end{lemma}

\begin{proof}
Fix an arbitrary $B\in \mathfrak{B}$. It is sufficient to show:
\[l(B)\leqslant(4\nu)^{2r-1}(r+1).\]

Since $S(B)\leqslant r$, we have $|B|\leqslant r+1$. Any set $C\in \mathfrak{B}$ that intersects with $B$ contains one of the $(r+1)$ elements of $B$. By Lemma \ref{lemma:estim_2},
for each element $\boldsymbol{t}_0\in B$ there are at most $(4\nu)^{2r-1}$ sets $C\in \mathfrak{B}$ with $\boldsymbol{t}_0\in C$. So the total number of sets $C\in \mathfrak{B}$ that intersect with $B$ is not greater than $(4\nu)^{2r-1}(r+1)$.
\end{proof} 
 
\textit{Remarks}. 
1) $L$ is a characteristic of interaction in the physical system.

2) We did not aim to obtain the most accurate estimate of $L$. However, it is not hard to get more accurate estimates in particular cases. For example, in the Ising, Potts and $n$-vector models $L=4\nu-1$ because in these models only neighboring points in $\mathbb{Z}^\nu$ interact, i.e. each set in $\mathfrak{B}$ is a set of two neighboring points.  

\subsection{Convergence Theorem}

\begin{definition}
Define $\lambda_0=\dfrac{1}{50L(8\nu)^{2r}}$, where $L$ was introduced in Definition \ref{L}.2.
\label{def:lambda_0}
\end{definition}

Clearly, $\lambda_0$ depends only on $\nu$ and $r$.

\begin{notation}
For any $N\in\mathbb{N}$ denote a cube in $\mathbb{Z}^\nu$:
\[\Lambda_N=\{\boldsymbol{t}=(t_1,\ldots,t_\nu)\in\mathbb{Z}^\nu\,\Big|\, |t_i|\leqslant N\text{ for each } i=1,\ldots,\nu\}.\]
\end{notation}

\begin{notation}
\begin{enumerate}[1)]
Suppose $Q$ is a finite subset of $\mathbb{Z}^\nu$ and its size $q=S(Q)$.

\item Denote $\mathbb{N}_Q=\{N\in\mathbb{N}\mid Q\subseteq \Lambda_N\}$.
For any $N\in\mathbb{N}_Q$:
\medskip
\begin{itemize}
\item $(\Lambda_N, U_{\Lambda_N},P_{\Lambda_N})$ is the finite interaction model with characteristics $\lambda$, $r$, $\Phi$, $\{P_{\boldsymbol{t}}\mid \boldsymbol{t}\in\mathbb{Z}^\nu\}$ from Definition  \ref{def:finite_model};

\item for brevity denote $U_N=U_{\Lambda_N}$ and $P_N=P_{\Lambda_N}$.
\end{itemize}
\medskip
\item For any $A\in\Sigma_Q$ define:
\begin{equation}
P_{\lambda,Q}(A)=\lim_{N\rightarrow\infty}P_N(A).
\label{eq:lim}
\end{equation}
\end{enumerate}
\end{notation}

\begin{theorem} [Convergence Theorem]
Suppose $0\leqslant\lambda\leqslant\lambda_0$. Then 
the following hold.
\begin{enumerate}[1)]
\item For any finite set $Q\subset\mathbb{Z}^\nu$ and any $A\in\Sigma_Q$ the limit in \eqref{eq:lim} exists.

\item $P_{\lambda,Q}$ is a probability measure on $(\Omega,\Sigma_Q)$.
\end{enumerate}
\label{theorem:Gibbs_measure}
\end{theorem}

This theorem will be proven in Sections \ref{section:estimations} and \ref{section:main_proof}.

\begin{corollary}
Suppose $0\leqslant\lambda\leqslant\lambda_0$. There exists a unique probability measure $P_\lambda$ on $(\Omega,\Sigma)$ such that for any finite set $Q\subset\mathbb{Z}^\nu$ and any $A\in\Sigma_Q$:
\[P_\lambda(A)=P_{\lambda,Q}(A).\]
\label{cor:Kolmogorov}
\end{corollary}

\begin{proof}
For any integer $k>1$, distinct $\boldsymbol{t}_1,\ldots,\boldsymbol{t}_k\in\mathbb{Z}^\nu$ and $F_i\in\Sigma_{\boldsymbol{t}_i}$ $(i=1,\ldots,k)$ denote 
\[P_{\lambda,\boldsymbol{t}_1\ldots \boldsymbol{t}_k}(F_{1}\times\ldots\times F_{k})=P_{\lambda,\{\boldsymbol{t}_1,\ldots ,\boldsymbol{t}_k\}}(\mathcal{F}_1\cap\ldots\cap\mathcal{F}_k),\]
where $\mathcal{F}_i=\{\omega\in\Omega\mid\omega(\boldsymbol{t}_i)\in F_i\}$, $i=1,\ldots,k$.
\medskip
\\
By Theorem \ref{theorem:Gibbs_measure}.2), each $P_{\lambda,\{\boldsymbol{t}_1,\ldots ,\boldsymbol{t}_k\}}$ is a probability measure on $(\Omega,\Sigma_{\{\boldsymbol{t}_1,\ldots ,\boldsymbol{t}_k\}})$. Then each $P_{\lambda,\boldsymbol{t}_1\ldots \boldsymbol{t}_k}$ is a probability measure on $(\Omega_{\boldsymbol{t}_1}\times\ldots\times\Omega_{\boldsymbol{t}_k},\Sigma_{\boldsymbol{t}_1\ldots\boldsymbol{t}_k} )$ that was introduced in Definition \ref{def:initial}.1).
These probabilities satisfy the following two consistency conditions.

\begin{enumerate}[1)]
\item For any distinct $\boldsymbol{t}_1,\ldots,\boldsymbol{t}_k\in\mathbb{Z}^\nu$, permutation  
$\pi$ of $\{1,2,\ldots,k\}$, and $F_i\in\Sigma_{\boldsymbol{t}_i}$ $(i=1,\ldots,k)$:
\[P_{\lambda,\boldsymbol{t}_{\pi(1)}\ldots \boldsymbol{t}_{\pi(k)}} (F_{\pi (1)}\times \ldots \times F_{\pi (k)}) 
=P_{\lambda,\boldsymbol{t}_1\ldots \boldsymbol{t}_k}(F_{1}\times\ldots\times F_{k}).\]

\item For any distinct $\boldsymbol{t}_1,\ldots,\boldsymbol{t}_k,\boldsymbol{t}_{k+1}\in\mathbb{Z}^\nu$ and $F_i\in\Sigma_{\boldsymbol{t}_i}$ $(i=1,\ldots,k)$:
\[P_{\lambda,\boldsymbol{t}_{1}\dots \boldsymbol{t}_{k}}(F_{1}\times\dots\times F_{k})
=P_{\lambda,\boldsymbol{t}_{1}\dots \boldsymbol{t}_{k}\boldsymbol{t}_{k+1}}\big(F_{1}\times \dots \times F_{k}\times\mathbb{R}\big).\]
\end{enumerate}

Hence the corollary follows from the Kolmogorov extension theorem.
\end{proof}

The probability measure $P_{\lambda}$ is called the \textbf{Gibbs measure}.

\subsection{Infinite Interaction Model}
\begin{definition}
Suppose $0\leqslant\lambda\leqslant\lambda_0$. 
A \textbf{infinite interaction model} with characteristics $\lambda,r,\Phi$, $\{P_{\boldsymbol{t}}\mid \boldsymbol{t}\in\mathbb{Z}^\nu\}$ is the probability space
$(\Omega,\Sigma,P_{\lambda}),$
where $P_{\lambda}$ is the probability measure from Corollary \ref{cor:Kolmogorov}.

The probability space $(\Omega,\Sigma,P_{\lambda})$ is also called the \textbf{thermodynamic limit} or \textbf{macroscopic limit} of the finite interaction models $(\Lambda_N, U_N,P_N)$ as $N\rightarrow\infty$.
\label{def:infinite_model} 
\end{definition}

Next we show that several well-known models can be considered as particular cases of the interaction model.

\subsection{Potts model as a particular case of interaction model}

The standard Potts model is based on the interaction energy (also called Hamiltonian):
\[U_\Lambda =\sum\limits_{\substack{\boldsymbol{s},\boldsymbol{t}\in \Lambda\\\|\boldsymbol{s}-\boldsymbol{t}\|=1}}J_{\boldsymbol{s}\boldsymbol{t}}\delta(X_{\boldsymbol{s}},
X_{\boldsymbol{t}}),\]
where $J_{\boldsymbol{s}\boldsymbol{t}}$ are coupling constants and the random variables $X_{\boldsymbol{t}}$ (interpreted as site colors) take values $1,2,\ldots, q$ with equal probabilities;
\[\delta(x,y)=\left\lbrace
\begin{array}{cc}
1\text{ if }x=y,
\\
0\text{ if }x\neq y.
\end{array}
\right.\]

Applying Definition \ref{def:parameters} we take an arbitrary $\nu\geqslant1$, an arbitrary $\lambda\geqslant0$, $r=1$, and the probabilities $\{P_{\boldsymbol{t}}\mid \boldsymbol{t}\in\mathbb{Z}^\nu\}$ such that for any Borel set $A$:
\[P_{\boldsymbol{t}}(A)=\dfrac{|A\cap\{1,2,\ldots, q\}|}{q},\]
where the numerator is the cardinality of the set $A\cap\{1,2,\ldots, q\}$.

Since $r=1$, then 
$\mathfrak{B}=\{ \{\boldsymbol{s},\boldsymbol{t}\}\subset\mathbb{Z}^\nu \;|\; \|\boldsymbol{s}-\boldsymbol{t}\|=1\}$ and
for any $B\in\mathfrak{B}$ with $B=\{\boldsymbol{s},\boldsymbol{t}\}$:
\[\Phi_B(\omega)=\lambda K_{\boldsymbol{s}\boldsymbol{t}}\delta(\omega
(\boldsymbol{s}),
\omega(\boldsymbol{t})),\]
where $|K_{\boldsymbol{s}\boldsymbol{t}}|\leqslant1$.

When $0\leqslant\lambda\leqslant\lambda_0$ and $N\rightarrow\infty$ (i.e. $\Lambda\rightarrow\mathbb{Z}^\nu$), the Potts model becomes a particular case of the infinite interaction model.

\subsection{Ising model as a particular case of interaction model}
The Ising model is equivalent to the Potts model for $q=2$.
The Ising model is based on the interaction energy:
\[U_\Lambda =\sum\limits_{\substack{\boldsymbol{s},\boldsymbol{t}\in \Lambda\\\|\boldsymbol{s}-\boldsymbol{t}\|=1}}J_{\boldsymbol{s}\boldsymbol{t}}X_{\boldsymbol{s}}
X_{\boldsymbol{t}}
-\sum\limits_{\boldsymbol{t}\in \Lambda} h_{\boldsymbol{t}} X_{\boldsymbol{t}},\]
where $J_{\boldsymbol{s}\boldsymbol{t}}, h_{\boldsymbol{t}}$ are  constants and the random variables $X_{\boldsymbol{t}}$  take values $\pm 1$ with equal probabilities.

To eliminate the second sum in $U_{\Lambda}$ (representing an external magnetic field) we apply a single Gibbs modification described in \cite{K17}, pg.343-344, and get the probabilities $\{P_{\boldsymbol{t}}\mid \boldsymbol{t}\in\mathbb{Z}^\nu\}$ such that for any Borel set $A$:
\begin{equation}
P_{\boldsymbol{t}}(A)=
\dfrac{1}{e^{h_{\boldsymbol{t}}}+e^{-h_{\boldsymbol{t}}}}
\sum\limits_{\substack{x\in A\\x=\pm1}}e^{-xh_{\boldsymbol{t}}}.
\label{eq:Borel}
\end{equation}

Applying Definition \ref{def:parameters} we take an arbitrary $\nu\geqslant1$, an arbitrary $\lambda\geqslant0$, the probabilities $\{P_{\boldsymbol{t}}\mid \boldsymbol{t}\in\mathbb{Z}^\nu\}$ given by \eqref{eq:Borel}, and $r=1$. 

Then 
$\mathfrak{B}=\{ \{\boldsymbol{s},\boldsymbol{t}\}\subset\mathbb{Z}^\nu \;|\; \|\boldsymbol{s}-\boldsymbol{t}\|=1\}$ and
for any $B\in\mathfrak{B}$ with $B=\{\boldsymbol{s},\boldsymbol{t}\}$:
\[\Phi_B(\omega)=\lambda K_{\boldsymbol{s}\boldsymbol{t}}\omega
(\boldsymbol{s})
\omega(\boldsymbol{t}),\]
where $|K_{\boldsymbol{s}\boldsymbol{t}}|\leqslant1$.

When $0\leqslant\lambda\leqslant\lambda_0$ and $N\rightarrow\infty$ (i.e. $\Lambda\rightarrow\mathbb{Z}^\nu$), the Ising model becomes a particular case of the infinite interaction model.

\subsection{n-vector model}
In the $n$-vector model classical spins $\mathbf{s}_i$ ($i\in \mathbb{Z}^\nu$) are  $n$-dimensional vectors with unit length; the interaction energy of the $n$-vector model is given by:
 $$U_{\Lambda}=-J\sum\limits_{\substack{\boldsymbol{i},\boldsymbol{j}\in \Lambda\\\|\boldsymbol{i}-\boldsymbol{j}\|=1}}\mathbf{s}_{\boldsymbol{i}} \cdot\mathbf{s}_{\boldsymbol{j}},$$
where $\cdot$ denotes scalar product.

Our interaction model can be easily generalised to the case when $X_{\boldsymbol{t}}$ has values in $\mathbb{R}^n$: in Definition \ref{def:P_t} we take $\Omega_{\boldsymbol{t}}=\mathbb{R}^n$ and 
$\Sigma_{\boldsymbol{t}}$ as the Borel $\sigma$-algebra on $\mathbb{R}^n$. The construction of Definition \ref{def:initial} still applies producing $\Omega=\left\lbrace 
\omega\mid\omega:\mathbb{Z}^\nu\rightarrow \mathbb{R}^n\right\rbrace$. It can be shown that the $n$-vector model is a particular case of this generalised interaction model, similarly to the process in the previous two subsections.

Particular cases of the $n$-vector model are the Ising model ($n=1$), XY model ($n=2$), and Heisenberg model ($n=3$).

\section{Technical Definitions and Estimations}
\label{section:estimations}
The rest of the paper is devoted to the proof of the Convergence Theorem (Theorem \ref{theorem:Gibbs_measure}). 
For the rest of the paper we fix $\Omega,\Sigma$,  characteristics $\lambda,r,\Phi$, $\{P_{\boldsymbol{t}}\mid \boldsymbol{t}\in\mathbb{Z}^\nu\}$, a finite set $Q\subset\mathbb{Z}^\nu$ with size $q=S(Q)$, and an event $A\in\Sigma_Q$. We will consider only integers $N\in\mathbb{N}_Q$.

Section \ref{section:estimations} contains some technical concepts and lemmas, which lead to the proof of the Convergence Theorem in Section \ref{section:main_proof}.

\subsection{Semi-Invariants}
 Denote $\langle\cdot,\ldots,\cdot\rangle_0$ a semi-invariant with respect to the probability measure $P_0$. The definition and properties of semi-invariants are described in literature, see for example \cite{Ma80} and references in it. 
 
\begin{notation}
For a random variable $Y$ on $(\Omega,\Sigma)$ and a sequence $\gamma=(B_1,\ldots,$ $B_n)$ of elements of  $\mathfrak{B}$ denote 
\[\langle Y,\Phi_{\gamma}\rangle_0=\langle Y,\Phi_{B_1},\Phi_{B_2},\ldots,\Phi_{B_n}\rangle_0.\]
If $n=0$, then $\langle Y,\Phi_{B_1},\ldots,\Phi_{B_n}\rangle_0=\langle Y\rangle_0$.
\end{notation}

\begin{definition}
A sequence $\gamma=(B_0,B_1, ...,B_n)$ of subsets of $\mathbb{Z}^\nu$ is called \textbf{connected} if the following graph is connected:
\begin{enumerate}[a)]
\item 
its set of \textbf{vertices} is $\{0,1,\ldots,n\}$; 
\item 
a pair of vertices $i,j$ is connected with an \textbf{edge} if and only if $B_i\cap B_j\neq\varnothing$. 
\end{enumerate} 
\end{definition}

\begin{lemma}
Suppose $B_1,\ldots,B_n\in\mathfrak{B}$ and the sequence $(Q,B_1, ...,B_n)$ is not connected. Then
\[\langle I_A,\Phi_{B_1},\ldots,\Phi_{B_n}\rangle_0=0.\]
\label{lemma:not_connected}
\end{lemma}
\begin{proof}
Proof can be found in \cite{MM91}, page 29 (property C). 
\end{proof}

\subsection{Families}

\begin{definition}
\label{multi}
\begin{enumerate}[1)]
\item A \textbf{family} (of elements of $\mathfrak{B}$) is a set of pairs \\$\Gamma=\{(C_1,n_1), \ldots,(C_k,n_k)\}$, where $C_1,\ldots,C_k$ are distinct elements of $ \mathfrak{B}$ and $n_i\geqslant 1$ for each $i=1,\ldots,k$. 

\item The number $n_i$ is called the \textbf{multiplicity} of element $C_i$ in the family $\Gamma$.

\item We denote the \textbf{length} of the family $\Gamma$ as 
$|\Gamma|=n_1+n_2+\ldots+n_k$ and \[\Gamma!=n_1!\cdot n_2!\cdot\ldots\cdot n_k! \]

\item $u_j(\Gamma)=\sum\limits_{j:C_i\cap C_j\neq\varnothing}n_j.$ 
Clearly, $u_j(\Gamma)\geqslant n_j$. 
\medskip

\item For a random variable $Y$ on $(\Omega,\Sigma)$ denote 
\[\langle Y,\Phi_{\Gamma}\rangle_0=\langle Y,\underbrace{\Phi_{C_1},\ldots,\Phi_{C_1}}_{n_1\,\text{times}},\ldots,\underbrace{\Phi_{C_k},\ldots,\Phi_{C_k}}_{n_k\,\text{times}}\rangle_0.\]

\item The family $\Gamma=\{(C_1,n_1), \ldots,(C_k,n_k)\}$ is called \textbf{Q-connected} if the sequence $(Q,C_1, \ldots,C_k)$ is connected.
\end{enumerate} 
\end{definition} 

\begin{definition}
We say that a sequence $\gamma=(B_1,\ldots,B_n)$ \textbf{reduces} to a family $\{(C_1,n_1), \ldots,(C_k,n_k)\},$
if $C_1,\ldots,C_k$ are the elements 
$B_1,\ldots,B_n$ written without repetitions, and each $n_i$ is the number of times that $C_i$ is repeated in $\gamma$.

Clearly, $n_1+\ldots+n_k=n$.
\end{definition}

\begin{lemma}
For each family $\Gamma$ of length $n$ there are $\dfrac{n!}{\Gamma !}$ sequences that reduce to $\Gamma$.
\label{lemma:family}
\end{lemma}
\begin{proof}
The lemma follows from the definitions and combinatorics.
\end{proof} 

\begin{lemma}
If a sequence $\gamma$ reduces to a family $\Gamma$, then
\[\langle Y,\Phi_{\gamma}\rangle_0=\langle Y,\Phi_{\Gamma}\rangle_0.\]
\label{lemma:seq_family}
\end{lemma}
\begin{proof}
The lemma follows from the fact that the value of a semi-invariant does not  does depend on the order of random variables. 
\end{proof}

In the following two subsections we produce some estimates, which will be used to prove convergence in Section \ref{section:main_proof}. We are not looking for most accurate estimates, since we are not interested in the speed of the convergence. 
 
\subsection{Estimating the Number of Families}

\begin{lemma}
Fix a natural number $n>1$.
The number of all $Q$-connected 
\medskip
\\
families $\Gamma$ of length $n$ is less than $2^{2q}\left(2(8\nu)^{2r}\right)^n$.
\label{lemma:estim_Gam}
\end{lemma}

\begin{proof}
Denote $*$ the operation of concatenation of two arbitrary sequences. Thus, for sequences $\alpha=(\alpha_0,\ldots,\alpha_m)$ and $\tau=(\tau_0,\ldots,\tau_l)$:
\[\alpha*\tau=(\alpha_0,\ldots,\alpha_m,\tau_0,\ldots,\tau_l).\]

We also denote:
\[\alpha*\Big|_j\tau=(\alpha_0,\ldots,\alpha_j,\tau_0,\ldots,\tau_l,\alpha_{j+1},\ldots,\alpha_m);\,j=1,\ldots,m,\]
which means inserting the sequence $\tau$ between two consecutive elements of $\alpha$. Clearly, 
\[\alpha*\Big|_m\tau=\alpha*\tau.\]

Denote $h=2q$ and $p=2r$. Fix the associated track $tr_Q$ of the set $Q$: 
\[tr_Q=(\boldsymbol{t}_{00},\ldots,\boldsymbol{t}_{0h}).\]

For each $Q$-connected family $\Gamma$ we will construct a sequence $Tr_{\Gamma}$ as follows.
Consider a $Q$-connected family $\Gamma=\{(C_1,n_1),\ldots,(C_k,n_k)\}$ of length $n$. By induction on $j$ $(j=0,1,\ldots,k)$ we will construct a track $Tr_j$ and a sequence of pairs $\widetilde{Tr}_j$.

\textit{Case} $j=0$.

Define 
\[Tr_0=tr_Q\text{ and }\widetilde{Tr}_0=\left((a_{00},\boldsymbol{t}_{00}),\ldots,(a_{0h},\boldsymbol{t}_{0h})\right),\text{ where }
a_{0i}=\left\lbrace\begin{array}{cc}
0\;\text{ if }\;\boldsymbol{t}_{0i}\notin Q,
\\
1\;\text{ if }\;\boldsymbol{t}_{0i}\in Q.
\end{array}
\right.\]

\textit{Inductive Step}.
Assume we have constructed a track $Tr_{j-1}$ and a sequence $\widetilde{Tr}_{j-1}$. Due to connectedness, 
\[(Q\cup C_1\cup\ldots\cup C_{j-1})\cap\left(\bigcup\limits_{i=j}^kC_i\right)\neq\varnothing.\]

Select $\boldsymbol{t}_{\beta_j}$ as the first element of the track $Tr_{j-1}$ that belongs to this intersection (but not the start element of $Tr_{j-1}$). Without loss of generality we can assume 
\[\boldsymbol{t}_{\beta_j}\in(Q\cup C_1\cup\ldots\cup C_{j-1})\cap C_j.\]

Choose an extended track $tr_j$ of the set $C_j$: 
$tr_j=(\boldsymbol{t}_{j0},\boldsymbol{t}_{j1},\ldots,\boldsymbol{t}_{jp}).$
Since the track is a closed path, we can assume it starts at the point $\boldsymbol{t}_{\beta_j}$: $\boldsymbol{t}_{j0}=\boldsymbol{t}_{\beta_j}$.
Define 
\[Tr_j=Tr_{j-1}*\Big|_{\beta_j}(\boldsymbol{t}_{j1},\ldots,\boldsymbol{t}_{jp}).\]

Denote $Tr'_{j-1}$ the sequence $\widetilde{Tr}_{j-1}$, where the element with index $\beta_j$ is changed from $(a,\boldsymbol{t})$ to $(n_j+1,\boldsymbol{t})$.

Define
\[\widetilde{Tr}_j
=Tr'_{j-1}*\Big|_{\beta_j}
\left((a_{j1},\boldsymbol{t}_{j1}),\ldots,(a_{jp},\boldsymbol{t}_{jp})\right),\text{ where }
a_{ji}=\left\lbrace\begin{array}{cc}
0\;\text{ if }\;\boldsymbol{t}_{ji}\notin C_j,
\\
1\;\text{ if }\;\boldsymbol{t}_{ji}\in C_j.
\end{array}
\right.\]

At the end of this process we take $Tr_\Gamma=\widetilde{Tr}_k$.

Next we estimate the number of sequences of the form $Tr_\Gamma$. A sequence $Tr_\Gamma=\widetilde{Tr}_k$ has $(1+h+kp)$ elements; each of these elements is a pair of a number and a point in $\mathbb{Z}^\nu$. The $(1+h)$ points of $tr_Q$ are fixed. For each of the remaining $kp$ points there are at most $2\nu$ choices, so there are at most $(2\nu)^{kp}\leqslant(2\nu)^{np}$ ways to choose second elements of the pairs in $Tr_\Gamma$. 

Next we estimate the number of choices for first elements of the pairs in $Tr_\Gamma$. 
There are $\binom{n+k-1}{k-1}$ ways to split $n$ into a sum $n=n_1+\ldots+n_k$.
If $n_1,\ldots,n_k$ are chosen, there are $\binom{h+kp}{k}$ ways to assign the numbers 
\\$(n_1+1),\ldots,(n_k+1)$ to first elements of some of the pairs in $Tr_\Gamma$. There are $2^{kp-k}$ ways to assign 0 and 1 to first elements of the remaining pairs (since the points of $tr_Q$ are fixed). So the total number of choices for first elements of the pairs in $Tr_\Gamma$ is at most
\[
\sum\limits_{k=1}^n {{n+k-1}\choose{k-1}}\cdot {{h+kp}\choose{k}}\cdot 2^{kp-k}
\leqslant 
\sum\limits_{k=1}^n 2^{n+k-1}\cdot2^{h+kp}\cdot2^{kp-k}
\]
\[
= 2^{n+h-1}\sum\limits_{k=1}^n 2^{2kp}
\leqslant 2^{n+h-1}\sum\limits_{j=1}^{2np}2^{j}
<2^{n+h-1}\cdot2^{2np+1}
=2^{h}\cdot2^{n(2p+1)}.
\]

So the total number of sequences of the form $Tr_\Gamma$ is less than 
\[(2\nu)^{np}\cdot2^{h}\cdot2^{n(2p+1)}
=2^h\left(2^{3p+1}\nu^p\right)^n
=2^{h}\left(2(8\nu)^{p}\right)^n
=2^{2q}\left(2(8\nu)^{2r}\right)^n.\]

In order to prove the same estimation for the number of $Q$-connected families $\Gamma$ with $|\Gamma|=n$ it is sufficient to show that a family $\Gamma$ can be restored from a sequence $Tr_{\Gamma}$ and it is unique (then each $Tr_{\Gamma}$ corresponds to only one $\Gamma$).

Consider one of the previously constructed sequences and denote it $Tr$. We will restore the family $\Gamma$ such that $Tr=Tr_{\Gamma}$. $Tr$ is a sequence of pairs that contains a sequence $\left((b_{00},\boldsymbol{t}_{00}),\ldots,(b_{0h},\boldsymbol{t}_{0h})\right)$ corresponding to the sequence $tr_Q=(\boldsymbol{t}_{00},\ldots,\boldsymbol{t}_{0h})$. 
We can assume that the sequence $Tr$ starts with the pair $(1,\boldsymbol{t}_{00})$.

 Denote $k$ the number of the pairs in $Tr$, whose first elements are greater than 1, and denote these first elements $m_1,\ldots,m_k$. Then $n_i=m_i-1$ $(i=1,\ldots,k)$, according to our construction.
So $\Gamma=\{(C_1,n_1)\ldots,(C_k,n_k)\}$. It remains to find the sets $C_j$ $(j=1,\ldots,k)$.
We will restore each set $C_j$ and each sequence $\widetilde{Tr}_{j-1}$ by induction on $k-j$.
\medskip

\textit{Case} $k-j=0$.
\medskip

Then $j=k$; $\widetilde{Tr}_{k}=Tr$. Find the last pair in $\widetilde{Tr}_{k}$, whose first element is greater than 1, and next $p$  pairs in $\widetilde{Tr}_{k}$. Then $\widetilde{Tr}_{k-1}$ is obtained from the sequence $\widetilde{Tr}_{k}$ by removing all these pairs and $C_k$ is the set of all second elements of these pairs that have 1 on the first place.

\textit{Inductive Step}. Assume we have restored sets $C_k,\ldots,C_{j+1}$ and a sequence $\widetilde{Tr}_{j}$. Find the last pair in $\widetilde{Tr}_{j}$, whose first element is greater than 1, and next $p$  pairs in $\widetilde{Tr}_{j}$. Then $\widetilde{Tr}_{j-1}$ is obtained from the sequence $\widetilde{Tr}_{j}$ by removing all these pairs and $C_j$ is the set of all second elements of these pairs that have 1 on the first place.

Clearly, this induction process produces the unique family \\  
$\Gamma=\{(C_1,n_1),\ldots,(C_k,n_k)\}$ with $Tr_{\Gamma}=Tr$.
\end{proof}

\subsection{Estimation of Semi-Invariants}

\begin{lemma}
For any $Q$-connected family $\Gamma=\{(C_1,n_1), \ldots,(C_k,n_k)\}$ of length $n$:
\[|\langle I_A,\Phi_\Gamma\rangle_0|\leqslant \frac{9}{2}P_0(A)\lambda^n (n+1)3^n\prod_{j=1}^k(u_j+1)^{n_j},\]
where each $u_j=u_j(\Gamma).$
\label{lemma:estim Mal}
\end{lemma}
\begin{proof}

Let us fix a sequence $\gamma=(B_1,\ldots,B_n)$ that reduces to $\Gamma$ (such a sequence always exists, according to Lemma \ref{lemma:family}). Denote $\gamma'=(Q,B_1,\ldots,B_n)$. 

For $i=1,2,\ldots,n$ denote $v_i$ the number of elements of $\gamma'$ that intersect with $B_i$. Denote $v_0$ the number of elements of $\gamma'$ that intersect with $Q$.

$I_A$ and all $\Phi_{B_i}$ are random variables on $(\Omega,\Sigma)$. 
According to Lemma \ref{lemma:seq_family} and Theorem 1 in \cite{MM91}, pg. 69,
\[|\langle I_A,\Phi_\Gamma\rangle_0|= |\langle I_A,\Phi_\gamma\rangle_0|\leqslant C_f\cdot\frac{3}{2} \prod_{i=0}^n(3v_i).\]

\[\text{Here }C_f=\max\Big|\langle I_A\cdot\prod_{i\in M_1}\Phi_{B_i}\rangle_0 \langle \prod_{i\in M_2}\Phi_{B_i}\rangle_0\ldots\langle \prod_{i\in M_l}\Phi_{B_i}\rangle_0\Big|,\]
where the maximum is taken over all partitions $\{\{0\}\cup M_1,M_2,\ldots,M_l\}$ of the set $\{0,1,2,\ldots,n\}$. 

Each $|\Phi_{B_i}|\leqslant\lambda$, so  $C_f\leqslant P_0(A)\lambda^n$. Hence
\[|\langle I_A,\Phi_\Gamma\rangle_0|\leqslant P_0(A)\lambda^n\cdot\frac{3}{2} \prod_{i=0}^n(3v_i)
=\frac{3}{2}P_0(A)\lambda^n3^{n+1}\prod_{i=0}^nv_i
=\frac{9}{2}P_0(A)\lambda^n3^{n}\prod_{i=0}^nv_i.\]

For $j=1,2,\ldots,k$ denote $w_j$ the number of elements of $\gamma'$ that intersect with $C_j$. Each  $C_j$ is repeated $n_j$ times in $\gamma$, so
\[\prod_{i=0}^nv_i
=v_0\prod_{i=1}^nv_i
=v_0\prod_{j=1}^k(w_j)^{n_j}.\]

Since $v_0\leqslant n+1$ and each $w_j\leqslant u_j+1$, we have:
\[\prod_{i=0}^nv_i
\leqslant (n+1)\prod_{j=1}^k(u_j+1)^{n_j} \text{ and}\]
\[|\langle I_A,\Phi_\Gamma\rangle_0|\leqslant \frac{9}{2}P_0(A)\lambda^n3^{n}(n+1)\prod_{j=1}^k(u_j+1)^{n_j}.\]
\end{proof}

\begin{lemma}
For any $Q$-connected family $\Gamma=\{(C_1,n_1), \ldots,(C_k,n_k)\}$ of length $n$: 
\[\prod_{j=1}^k (u_j+1)^{n_j}<(eL)^n \prod_{j=1}^k(n_j)^{n_j},\]
\label{Es 1}
where each $u_j=u_j(\Gamma).$
\end{lemma}

\begin{proof} 
$$u_j+1=\left(1+\dfrac{1}{u_j}\right)u_j$$
and 
\begin{equation}
\prod_{j=1}^k\left(u_j+1\right)^{n_j}
=\prod_{j=1}^k \left(1+\dfrac{1}{u_j}\right)^{n_j}
\prod_{j=1}^k u_j^{n_j}.
\label{eq:product_1}
\end{equation}

From Definition \ref{multi} we get $n_j\leqslant u_j$, so
\[\prod_{j=1}^k \left(1+\dfrac{1}{u_j}\right)^{n_j}\leqslant\prod_{j=1}^k \left(1+\dfrac{1}{n_j}\right)^{n_j}<\prod_{j=1}^k e=e^k\leqslant e^n.\]
Thus,
\begin{equation}
\prod_{j=1}^k \left(1+\dfrac{1}{u_j}\right)^{n_j}< e^n.
\label{eq:product_2}
\end{equation}

By Theorem 4.1 in \cite{K15} we have:
\begin{equation}
\label{eq ln L}
\sum_{j=1}^k\left(n_j \ln\frac{u_j}{n_j}\right)\leqslant n\ln L.
\end{equation}
Since 
\[\prod_{j=1}^k \left(\dfrac{u_j}{n_j}\right)^{n_j}= e^{\,\sum\limits_{j=1}^k \left(n_j \ln\frac{u_j}{n_j}\right)},\]
we imply from \eqref{eq ln L}:
\[\prod_{j=1}^k\left(\dfrac{u_j}{n_j}\right)^{n_j}\leqslant e^{n\ln L}=L^n.\]
So
\begin{equation}
\prod_{j=1}^k u_j^{n_j} \leqslant L^n\prod_{j=1}^k n_j^{n_j}.
\label{eq:ln}
\end{equation}

From \eqref{eq:product_1}, \eqref{eq:product_2} and \eqref{eq:ln} we get:
$$\prod_{j=1}^k\left(u_j+1\right)^{n_j}< e^nL^n\prod_{j=1}^k n_j^{n_j}= (eL)^n\prod_{j=1}^k n_j^{n_j}.$$
\end{proof}

\begin{lemma} 
Suppose $\Gamma$ is a $Q$-connected family of length $n$ and $n>3$. Then
\[|\langle I_A,\Phi_\Gamma\rangle_0|<P_0(A)\lambda^n(3e^2L)^n(n+1)\Gamma!\]
\label{lemma:estim Mal 2}
\end{lemma}
\begin{proof}
By Lemmas \ref{lemma:estim Mal} and \ref{Es 1}, 
\[|\langle I_A,\Phi_\Gamma\rangle_0|\leqslant \frac{9}{2}P_0(A)\lambda^n (n+1)3^n\prod_{j=1}^k(u_j+1)^{n_j}
<\frac{9}{2}P_0(A)\lambda^n (n+1)3^n(eL)^n\prod_{j=1}^kn_j^{n_j}.\]

By Stirling formula,
$q^q<\dfrac{1}{\sqrt{2\pi q}}\,q!\,e^q$ for any natural $q>0$. So
\begin{multline*}
|\langle I_A,\Phi_\Gamma\rangle_0|
<\frac{9}{2}P_0(A)\lambda^n (n+1)(3eL)^n
\prod_{j=1}^k \Big(\dfrac{1}{\sqrt{2\pi n_j}}\,n_j!\,e^{n_j}\Big)
\\
\leqslant \frac{9}{2}P_0(A)\lambda^n(3eL)^n(n+1)
\frac{1}{(2\pi)^{\frac{k}{2}}\sqrt{n_1n_2\ldots n_k}}\,e^n
\prod_{j=1}^k n_j!
\medskip
\\
=P_0(A)\lambda^n(3e^2L)^n(n+1)
\frac{9}{2(2\pi)^{\frac{k}{2}}\sqrt{n_1n_2\ldots n_k}}\,
\Gamma!
\end{multline*}
since $\Gamma!=
\prod\limits_{j=1}^k n_j!$. It remains to prove:
\[\frac{9}{2(2\pi)^{\frac{k}{2}}\sqrt{n_1n_2\ldots n_k}}<1.\]

Consider two cases.

\textit{Case 1}. $k=1$.
\\
Then $n_1=n\geqslant4$, $(2\pi)^{k/2}=\sqrt{2\pi}>2.5$ and 
\[\frac{9}{2(2\pi)^{\frac{k}{2}}\sqrt{n_1n_2\ldots n_k}}
=\frac{9}{2\sqrt{2\pi}\sqrt{n_1}}<
\frac{9}{2\times 2.5\times 2}<1.\]

\textit{Case 2}. $k\geqslant2$. 
\\
Then $(2\pi)^{k/2}\geqslant2\pi>6$ and 
 $\dfrac{9}{2(2\pi)^{k/2}\sqrt{n_1 n_2\ldots n_k}}
<\dfrac{9}{2\cdot6\cdot1}
<1.$
\end{proof}

\section{Proof of the Convergence Theorem }
\label{section:main_proof}
Here we use results from the previous section to prove Theorem \ref
{theorem:Gibbs_measure}. 

\begin{notation} 
Denote
\[J_A(N,n)=
\sum_{\Gamma}\frac{1}{\Gamma !}\langle I_A,\Phi_{\Gamma}\rangle_0,\]
where the sum is taken over all $Q$-connected families $\Gamma=\{(C_1,n_1), \ldots,(C_k,n_k)\}$ of length $n$ such that each $C_i\subseteq\Lambda_N$.
\label{notation:main}
\end{notation}

\begin{lemma}
Suppose $0\leqslant\lambda\leqslant\lambda_0$. Then 
\[P_N(A)=\sum_{n=0}^{\infty} J_A(N,n).\]
\label{lemma:family_sum}
\end{lemma}

\begin{proof}
Fix $\lambda$ such that $0\leqslant\lambda\leqslant\lambda_0$. 
We use the method from \cite{MM91}, page 34.  For real $x,z$ define 
\[f(x,z)=\ln \langle\exp(xI_A+ z U_N)\rangle_0.\]

By the definition of semi-invariants:

\[\frac{\langle I_A\,e^{zU_N}\rangle_0}{\langle e^{zU_N}\rangle_0}=\dfrac{\partial}{\partial x}f(x,z)\Big|_{x=0} =
\sum_{n=0}^{\infty} \frac{z^n}{n!} \langle I_A,
\underbrace{U_N,\ldots,U_N}_{n\,\text{times}}\rangle_0 
\]
\[=
\sum_{n=0}^{\infty} \frac{z^n}{n!} \langle I_A,
\sum_{B_1\in\mathfrak{B},B_1\subseteq\Lambda_N}\Phi_{B_1},\ldots,\sum_{B_n\in\mathfrak{B},B_n\subseteq\Lambda_N}\Phi_{B_n}\rangle_0 \]
\[=
\sum_{n=0}^{\infty} \frac{z^n}{n!}\sum_{\substack{B_1,\ldots,B_n,\\\text{each }B_i\in\mathfrak{B},B_i\subseteq\Lambda_N}}\langle I_A,
\Phi_{B_1},\ldots,\Phi_{B_n}\rangle_0.\]

Thus,
\[\frac{\langle I_A\,e^{zU_N}\rangle_0}{\langle e^{zU_N}\rangle_0}=\sum_{n=0}^{\infty} \dfrac{z^n}{n!}
\sum_{\gamma}\langle I_A,\Phi_{\gamma}\rangle_0,\]
where the inner sum is taken over all sequences $\gamma=\{B_1,\ldots,B_n\}$ such that each $B_i\in\mathfrak{B}$ and $B_i\subseteq\Lambda_N$.

Substituting $z=1$ and using the definition of $P_N$, we get:
\[
P_N(A)=\frac{\langle I_A\,e^{U_N}\rangle_0}{\langle e^{U_N}\rangle_0}=\sum_{n=0}^{\infty} \dfrac{1}{n!}
\sum_{\gamma}\langle I_A,\Phi_{\gamma}\rangle_0.
\]

By Lemma \ref{lemma:not_connected}, the inner sum can be taken only over sequences $\gamma= (B_1, ...,B_n)$ of elements of $\mathfrak{B}$ such that $(Q,B_1, ...,B_n)$ is connected and each $B_i\subseteq\Lambda_N$.

By Lemmas \ref{lemma:family} and \ref{lemma:seq_family}:
\[P_N(A)=\sum_{n=0}^{\infty} \dfrac{1}{n!}
\sum_{\gamma}\langle I_A,\Phi_{\gamma}\rangle_0=\sum_{n=0}^{\infty} \dfrac{1}{n!}
\sum_{\Gamma}\dfrac{n!}{\Gamma!}\langle I_A,\Phi_{\Gamma}\rangle_0=\sum_{n=0}^{\infty}
\sum_{\Gamma}\dfrac{1}{\Gamma !}\langle I_A,\Phi_{\Gamma}\rangle_0.\]

The first inner sum has only sequences $\gamma=(B_1,\ldots,B_n)$ such that \\$(Q,B_1,$ $\ldots,B_n)$ are connected. Therefore the last inner sum has only $Q$-connected families $\Gamma$ and
\[P_N(A)=\sum_{n=0}^{\infty} J_A(N,n).\] 
\end{proof}

\begin{lemma}
Suppose $0\leqslant\lambda\leqslant\lambda_0$ and $n>3$. Then 
\[\Big|J_A(N,n)\Big|\leqslant 2^{2q}P_0(A)0.9^n (n+1).\]
\label{lemma:main_0}
\end{lemma}
\begin{proof}
Suppose $0\leqslant\lambda\leqslant\lambda_0$ and $n>3$. 
$$\Big|J_A(N,n)\Big|\leqslant
\sum_{\Gamma}\dfrac{1}{\Gamma!}\Big|\langle I_A,\Phi_\Gamma\rangle_0\Big|,$$
where the sum is taken over all $Q$-connected families $\Gamma=\{(C_1,n_1), \ldots,(C_k,n_k)\}$ of length $n$ (we omit the restriction $C_i\subseteq\Lambda_N$).
By Lemma \ref{lemma:estim_Gam} the number of addends in this sum is less than $2^{2q}\left(2(8\nu)^{2r}\right)^n$. 

By Lemma \ref{lemma:estim Mal 2},
$$\Big|\langle I_A,\Phi_\Gamma\rangle_0\Big|\leqslant P_0(A)(n+1) \lambda^n(3e^2L)^n\Gamma!$$
and
\begin{multline*}
\Big|J_A(N,n)\Big|\leqslant \sum_{\Gamma}
\dfrac{1}{\Gamma!}P_0(A) (n+1)\lambda^n(3e^2L)^n\Gamma!
=\sum_{\Gamma}
P_0(A)(n+1) \lambda^n(3e^2L)^n
\\
\leqslant
2^{2q}\left(2(8\nu)^{2r}\right)^nP_0(A)(n+1)\lambda^n(3e^2L)^n
=2^{2q}P_0(A)(n+1)\left[\lambda\cdot6e^2L(8\nu)^{2r}\right]^n.
\end{multline*}

Since $0\leqslant\lambda\leqslant\lambda_0$, then by the Definition \ref{def:lambda_0} of $\lambda_0$ we have:
\[\lambda\cdot6e^2L(8\nu)^{2r}\leqslant\lambda_0\cdot6e^2L(8\nu)^{2r}
=\dfrac{6e^2}{50}
< 0.9\text{ and }\]
\[\Big|J_A(N,n)\Big|\leqslant 2^{2q}P_0(A)(n+1)0.9^n.\]
\end{proof}

\begin{lemma}
Suppose $0\leqslant\lambda\leqslant\lambda_0$. Then the series
\[\sum_{n=0}^{\infty} 
J_A(N,n)\]
converges absolutely and uniformly for all $N\in\mathbb{N}_Q$.
\label{lemma:main_1}
\end{lemma}
\begin{proof}
This follows from Lemma \ref{lemma:main_0}, since the series $\sum\limits_{n=0}^{\infty} 0.9^n (n+1)$ converges.
\end{proof}

\begin{lemma}
Suppose $0\leqslant\lambda\leqslant\lambda_0$. Then for any $n$ the following limit exists:
\[\lim\limits_{N\rightarrow\infty}J_A(N,n).\]
\label{lemma:main_2}
\end{lemma}
\begin{proof}
Denote $\boldsymbol{0}=(0,\ldots,0)$, the origin in $\mathbb{Z}^\nu$, 
and $d=\min\{\|\boldsymbol{t}-\boldsymbol{0}\|\;\big|\boldsymbol{t}\in Q\}$; $d$ is the distance of the set $Q$ from the origin. 

For any $n$ denote $M_n=r(n+1)+q+d$. 
We will prove that for any $N\geqslant M_n$:
\begin{equation}
J_A(N,n)=J_A(M_n,n).
\label{eq:stability}
\end{equation}
Then for any fixed $n$, $\lim\limits_{N\rightarrow\infty}J_A(N,n)$ exists and
$\lim\limits_{N\rightarrow\infty}J_A(N,n)=J_A(M_n,n).$

\begin{center}
Proof of \eqref{eq:stability}
\end{center}

Consider any $N\geqslant M_n$. Then $\Lambda_{M_n}\subseteq\Lambda_N$, which implies the following.

\begin{multline}
\text{If }\Gamma=\{(C_1,n_1), \ldots,(C_k,n_k)\}\text{ is a }Q\text{-connected family of length }n
\\ 
\text{ such that each }
C_i\subseteq\Lambda_{M_n},\text{ then each }C_i\subseteq\Lambda_{N}.
\label{eq:connect_1}
\end{multline}

It remains to prove:
\begin{multline}
\text{If }\Gamma=\{(C_1,n_1), \ldots,(C_k,n_k)\}\text{ is a }Q\text{-connected family of length }n
\\ 
\text{ such that each }
C_i\subseteq\Lambda_{N},\text{ then each }C_i\subseteq\Lambda_{M_n}.
\label{eq:connect_2}
\end{multline}

Then from \eqref{eq:connect_1} and \eqref{eq:connect_2} we imply that $J_A(N,n)$
 and $J_A(M_n,n)$ are sums over the same set of families $\Gamma$. Hence $J_A(N,n)=J_A(M_n,n)$.

\begin{center}
Proof of \eqref{eq:connect_2}
\end{center}

Consider a $Q$-connected family $\Gamma=\{(C_1,n_1), \ldots,(C_k,n_k)\}$ of length $n$ such that each $C_i\subseteq\Lambda_{N}.$
Fix $i=1,\ldots,k$. To prove that $C_i\subseteq\Lambda_{M_n}$, we fix an arbitrary $\boldsymbol{t}\in C_i$ and prove that $\boldsymbol{t}\in \Lambda_{M_n}$.

For some $\boldsymbol{t}_0\in Q$, $d=\|\boldsymbol{t}_0-\boldsymbol{0}\|$. 
Since $(Q,C_1,\ldots,C_n)$ is connected, there are a subsequence $(C_{j_1},\ldots,C_{j_m})$ of the sequence $(C_1,\ldots,C_k)$ and points $\boldsymbol{t}_1,\boldsymbol{t}_2,\ldots,\boldsymbol{t}_m,\boldsymbol{t}_{m+1}$ such that
\[\boldsymbol{t}_1\in C_i\cap C_{j_1},\;\boldsymbol{t}_2\in C_{j_1}\cap C_{j_2},\;\ldots,\boldsymbol{t}_m\in C_{j_{m-1}}\cap C_{j_m},
\;\text{and }\boldsymbol{t}_{m+1}\in C_{j_m}\cap Q.\]

Then $m\leqslant k\leqslant n$ and
\begin{multline*}
\|\boldsymbol{t}-\boldsymbol{0}\|
\leqslant
\|\boldsymbol{t}-\boldsymbol{t}_1\|+
\|\boldsymbol{t}_1-\boldsymbol{t}_2\|+\ldots+\|\boldsymbol{t}_m-\boldsymbol{t}_{m+1}\|+
\|\boldsymbol{t}_{m+1}-\boldsymbol{t}_0\|+
\|\boldsymbol{t}_0-\boldsymbol{0}\|
\\
\leqslant r(m+1)+q+d
\leqslant r(n+1)+q+d=M_n.
\end{multline*}

So $\|\boldsymbol{t}-\boldsymbol{0}\|\leqslant M_n$ and $\boldsymbol{t}\in \Lambda_{M_n}$.
\end{proof}
\begin{center}
\textbf{Proof of Theorem \ref
{theorem:Gibbs_measure}
}
\end{center}

\begin{proof}
1. Suppose $0\leqslant\lambda\leqslant\lambda_0$.
By Lemma \ref
{lemma:family_sum}, 
\begin{equation}
P_N(A)=\sum_{n=0}^{\infty} J_A(N,n).
\label{eq:J_sum}
\end{equation}

Due to Lemmas \ref{lemma:main_1}, \ref{lemma:main_2} and a property of uniform convergence, $\lim\limits_{N\rightarrow\infty}P_N(A)$ exists and 
\begin{equation}
P_{\lambda,Q}(A)=\lim\limits_{N\rightarrow\infty}P_N(A)=\sum_{n=0}^{\infty}\lim\limits_{N\rightarrow\infty} J_A(N,n).
\label{eq:limit}
\end{equation}

2. To prove that $P_{\lambda}$ is a probability measure on $(\Omega,\Sigma_Q)$ we check three probability axioms.

\[P_{\lambda,Q}(\varnothing)=\lim\limits_{N\rightarrow\infty}P_N(\varnothing)=\lim\limits_{N\rightarrow\infty} 0=0.\]

\[P_{\lambda,Q}(\Omega)=\lim\limits_{N\rightarrow\infty}P_N(\Omega)=\lim\limits_{N\rightarrow\infty} 1=1.\]

To complete the proof it remains to check the axiom of $\sigma$-additivity.
Consider a sequence of disjoint events $A_i\in\Sigma_Q$ $(i=1,2,\ldots)$ and denote $D=\bigcup\limits_{i=1}^\infty A_i$.
By  \eqref{eq:J_sum},
\begin{equation}
P_N(D)=\sum_{i=1}^{\infty}P_N(A_i)
=\sum_{i=1}^{\infty}
\sum_{n=0}^{\infty}
J_{A_i}(N,n).
\label{eq:last_series}
\end{equation}

By Lemma \ref{lemma:main_0}, for $n>3$ and $i=1,2,\ldots,$
\[|J_{A_i}(N,n)|\leqslant 2^{2q}P_0(A_i)0.9^n(n+1).\]

Since for $0<x<1$,
$\sum\limits_{n=0}^{\infty}x^n(n+1)=\dfrac{1}{(1-x)^2},$
we have:
\begin{multline*}
\sum_{i=1}^{\infty}
\sum_{n=0}^{\infty}2^{2q}P_0(A_i)0.9^n(n+1)
=2^{2q}\sum_{i=1}^{\infty}
P_0(A_i)
\sum_{n=0}^{\infty}
0.9^n(n+1)
\medskip
\\
=2^{2q}\sum_{i=1}^{\infty}
P_0(A_i)\frac{1}{(1-0.9)^2}
=2^{2q}100\sum_{i=1}^{\infty}
P_0(A_i)=2^{2q}100P_0(D).
\end{multline*}

So the series on the right-hand side of \eqref{eq:last_series} converges absolutely and uniformly for all $N\in\mathbb{N}_Q$. Taking a limit in \eqref{eq:last_series} and using \eqref{eq:limit}, we get:
\[P_{\lambda,Q}(D)
=\lim\limits_{N\rightarrow\infty}P_N(D)
=\sum_{i=1}^{\infty}
\sum_{n=0}^{\infty}
\lim\limits_{N\rightarrow\infty}J_{A_i}(N,n)
=\sum_{i=1}^{\infty}P_{\lambda,Q}(A_i).\]

This proves that 
\[P_{\lambda,Q}\left(\bigcup\limits_{i=1}^\infty A_i\right)
=\sum_{i=1}^{\infty}P_{\lambda,Q}(A_i).\]
\end{proof}

\bibliographystyle{plain}
\bibliography{ilias}

\begin{thebibliography}{10}

\bibitem{B08}
R.~J. Baxter.
\newblock {\em Exactly Solved Models in Statistical Mechanics}.
\newblock Dover Publications, 2008.
\newblock ISBN-10: 0486462714.

\bibitem{D68}
R.L. Dobrushin.
\newblock Gibbsian random fields for lattice system with pairwise interactions.
\newblock {\em Funct. Anal. Appl.}, 77:292--301, 1968.

\bibitem{FV17}
S.~Friedli and Y.~Velenik.
\newblock {\em Statistical Mechanics of Lattice Systems: a Concrete
  Mathematical Introduction}.
\newblock Cambridge University Press, 2017.
\newblock ISBN-10: 1107184827.

\bibitem{Geor}
H.-O. Georgii.
\newblock {\em Gibbs Measures and Phase Transitions}.
\newblock De Gruyter, 2011.

\bibitem{K15}
F.~Kachapova and I.~Kachapov.
\newblock Convergence of renormalization group transformations of {G}ibbs
  random field.
\newblock {\em Journal of Mathematics and Statistics}, 12(3):135--151, 2016.
\newblock \url{http://thescipub.com/PDF/jmssp.2016.135.151.pdf}.

\bibitem{K17}
F.~Kachapova and I.~Kachapov.
\newblock Interaction model in statistical mechanics.
\newblock {\em Journal of Mathematics and Statistics}, 13(4):339--346, 2017.
\newblock \url{http://thescipub.com/PDF/jmssp.2016.135.151.pdf}.

\bibitem{K77}
I.A. Kashapov.
\newblock Structure of ground states in three-dimensional {I}sing model with
  three-step interaction.
\newblock {\em Theor. Math. Phys.}, 33(1):110--118, 1977.
\newblock \url{https://doi.org/10.1007/BF01039015}.

\bibitem{MM91}
V.~A. Malyshev and R.~A. Minlos.
\newblock {\em Gibbs Random Fields: Cluster Expansions}.
\newblock Kluwer Academic Publishers, 1991.
\newblock ISBN-10: 079230232X,\url{DOI: 10.1007/978-94-011-3708-9}.

\bibitem{Ma80}
V.A. Malyshev.
\newblock Cluster expansions in lattice models of statistical physics and
  quantum field theory.
\newblock {\em Russian Math. Surveys}, 35(2):1--62, 1980.
\newblock
  \url{http://iopscience.iop.org/article/10.1070/RM1980v035n02ABEH001622/meta}.

\bibitem{Pres}
C.~Preston.
\newblock {\em Random Fields}.
\newblock Springer-Verlag, 1976.

\bibitem{Rue}
D.~Ruelle.
\newblock {\em Statistical Mechanics: Rigorous Results}.
\newblock World Scientific Publishing Co., 1999.

\bibitem{Yan}
Y.~Yang, B.~Teng, H.~Yang, and H.~Cui.
\newblock Investigation of probability theory on {I}sing models with different
  four-spin interactions.
\newblock {\em Physica A: Statistical Mechanics and Its Applications},
  483:243--349, 2017.
\newblock \url{https://doi.org/10.1016/j.physa.2017.04.176}.

\end{thebibliography}

\end{document}